\documentclass[a4paper]{amsart}
\usepackage{amsmath,amsthm,amssymb,latexsym,epic,bbm,comment}
\usepackage{graphicx,enumerate,stmaryrd}
\usepackage[all,2cell]{xy}
\xyoption{2cell}

\usepackage[active]{srcltx}
\usepackage[parfill]{parskip}
\UseTwocells

\newtheorem{theorem}{Theorem}
\newtheorem{lemma}[theorem]{Lemma}

\newtheorem{corollary}[theorem]{Corollary}
\newtheorem{proposition}[theorem]{Proposition}

\newcommand{\tto}{\twoheadrightarrow}

\font\sc=rsfs10
\newcommand{\cC}{\sc\mbox{C}\hspace{1.0pt}}

\font\scc=rsfs7
\newcommand{\ccC}{\scc\mbox{C}\hspace{1.0pt}}

\begin{document}

\title[Special modules over positively based algebras]
{Special modules over positively based algebras}
\author{Tobias Kildetoft and Volodymyr Mazorchuk}

\begin{abstract}
We use the Perron-Frobenius Theorem to define, study 
and, in some sense, classify special simple modules 
over arbitrary finite dimensional positively based algebras. For group
algebras of finite Weyl groups with respect to the Kazhdan-Lusztig basis,
this agrees with Lusztig's notion of a special module introduced in \cite{Lu1}.
\end{abstract}

\maketitle

\section{Introduction and description of the results}\label{s1}

In \cite{Lu1,Lu2}, Lusztig used combinatorics of generic degrees to define and study 
a certain class of  Weyl group representations which he called {\em special}. These
representations play an important role in the study of Kazhdan-Lusztig left cell representations,
see \cite{KL,Lu3,Ge1}.

The present paper proposes an approach to the definition and study of special modules
for arbitrary finite dimensional positively based algebras. By the latter we mean 
an algebra over a subfield $\Bbbk$ of the complex numbers with a fixed basis such that 
all structure constants with respect to this basis are non-negative real numbers.
Examples of such algebras include group algebras and semigroup algebras with the standard
basis, but also group algebras of Coxeter groups and the corresponding Hecke algebras
with respect to the Kazhdan-Lusztig basis.

Our approach is motivated by some techniques originating in the abstract $2$-representation 
theory of finitary $2$-categories developed in the series \cite{MM1,MM2,MM3,MM4,MM5,MM6}
of papers. A major emphasis in these papers was made on the study of so-called 
{\em cell $2$-representations}. On the level of the Grothendieck group, a cell 
$2$-representation becomes a based module over some finite-dimensional positively based algebra
with various nice properties. For example, for the $2$-category of Soergel bimodules
over the coinvariant algebra of a finite Coxeter group, the Grothendieck group 
level of a cell $2$-representation is exactly the Kazhdan-Lusztig left cell module.
In this sense, abstract representation theory of finitary $2$-categories 
proposes a generalization of the situation mentioned in the previous paragraph.

A crucial technical tool in this study of cell $2$-representations turned
out to be the classical Perron-Frobenius Theorem from \cite{Fr1,Fr2,Pe}, 
see for example applications of this theorem in \cite{MM4,MM5,MM6}.
This theorem also plays a very important role in some further developments,
see for example \cite{CM,MZ,Zi}. The main point of the present paper is the observation that
one can use  the Perron-Frobenius Theorem to define special modules for arbitrary 
transitive $2$-representations of finitary $2$-categories. In fact, the definition does
not require any properties of the $2$-layer of the structure and hence can be 
formulated for the general setup of positively based algebras.

Given an algebra $A$ with a positive basis $\mathbf{B}$, one can define the notions
of {\em left, right and two-sided} orders and cells, similarly to the definition of
Green's orders and relations for semigroups (see \cite{Gr}) or multisemigroups
(see \cite{KuMa}), or Kazhdan-Lusztig orders and cells in Kazhdan-Lusztig theory (see \cite{KL}). 
This can be used to define left cell modules for $A$. Such a module, denoted
$C_{\mathcal{L}}$, where $\mathcal{L}$ is a left cell, is a based module with a fixed
basis $\mathbf{B}_{\mathcal{L}}$ that can be canonically identified with 
a subset of $\mathbf{B}$. Now, for any element $a\in A$ which can be
written as a linear combination of all elements in $\mathbf{B}$ with positive real coefficients,
all entries of the matrix of $a$ in the basis $\mathbf{B}_{\mathcal{L}}$ are positive real 
numbers. This allows us to use the Perron-Frobenius Theorem, namely, uniqueness and simplicity of 
the Perron-Frobenius eigenvalue for $a$, to define the {\em special} subquotient
of $C_{\mathcal{L}}$, which is a certain simple module that appears in $C_{\mathcal{L}}$ with
multiplicity one. The original definition depends both on the choice of $a$ and
$\mathcal{L}$. However, in Subsection~\ref{s5.4} we show that the resulting special
module is independent of the choice of $a$. Further, in Subsection~\ref{s5.5} we show
that it is also independent of the choice of $\mathcal{L}$ inside a fixed two-sided cell.
We give a complete classification of special modules in Corollary~\ref{corclass} by showing
that there is a one-to-one correspondence between special modules and idempotent 
two-sided cells.

The paper is organized as follows: In Section~\ref{s2} we give the definition of
positively based algebras and list several classical examples. In Section~\ref{s3} we
describe basic properties and combinatorics for positively based algebras. 
In Section~\ref{s4} we recall the Perron-Frobenius Theorem. In Section~\ref{s5}
we introduce the notion of special modules and study  basic properties of such modules, 
in particular the independence properties mentioned above. In Section~\ref{s6} we describe 
special modules for our three principal examples: group algebras (in the standard basis), 
semigroup algebras, and group algebras of finite Weyl groups with respect to the Kazhdan-Lusztig
basis. In particular, we show that, in the latter case, our notion of a special module
coincides with Lusztig's definition of special modules from \cite{Lu1}. 
In Section~\ref{s7} we obtain some further properties of special modules.
In Section~\ref{s8} we define and study the notion of the apex.
Finally, in Section~\ref{s9}, we consider special subquotients for arbitrary transitive 
$A$-modules and give a complete classification of special subquotients in terms of 
idempotent two-sided cells of $A$. As an application, we obtain an elementary explanation for the
fact that different left Kazhdan-Lusztig cells inside a given two-sided Kazhdan-Lusztig cell
are not comparable with respect to the Kazhdan-Lusztig left order. As another application,
we show that all Kazhdan-Lusztig two-sided cells are good in the sense of \cite{CM}.
\vspace{5mm}

\noindent
{\bf Acknowledgment.} This research was done during the postdoctoral stay of the first author
at Uppsala University which was supported by the Knut and Alice Wallenbergs Stiftelse.
The second author is partially supported by the Swedish Research Council.
We thank Meinolf Geck for very helpful discussions.
\vspace{5mm}

\section{Positively based algebras: definition and examples}\label{s2}

\subsection{Algebras with a positive basis}\label{s2.1}

Let $\Bbbk$ be a unital subring of the field $\mathbb{C}$ of complex numbers.
Let $A$ be a $\Bbbk$-algebra which is free of finite rank $n$ over $\Bbbk$. A $\Bbbk$-basis
$\mathbf{B}=\{a_i\,:\,i=1,2,\dots,n\}$ of $A$ will be called {\em positive} provided that 
all structure constants of $A$ with respect to this basis are non-negative real numbers, that is,
for all $i,j\in\{1,2,\dots,n\}$, we have
\begin{equation}\label{eq1}
a_i\cdot a_j =\sum_{k=1}^n \gamma_{i,j}^{(k)} a_k,\quad\text{ where }\quad
\gamma_{i,j}^{(k)}\in\mathbb{R}_{\geq 0}\quad\text{ for all }\quad i,j,k. 
\end{equation}
An algebra with a fixed positive basis is called a {\em positively based algebra}.

The above notion also makes perfect sense for infinite dimensional algebras. However, 
in this paper we restrict our study to algebras which are finitely generated over the base ring.
For interesting infinite dimensional examples, see \cite{Th}.

\subsection{Example I: group algebras}\label{s2.2}

Let $G$ be a finite group and $\Bbbk[G]$ the corresponding group algebra which consists of 
all elements of the form $\displaystyle\sum_{g\in G}c_g g$, where $c_g\in\Bbbk$. This algebra
is positively based with
respect to the {\em standard basis} $\mathbf{B}=\{g:g\in G\}$. In fact, all structure constants
with respect to this basis are either zero or one.

\subsection{Example II: semigroup algebras}\label{s2.3}

A straightforward generalization of the previous example is the following.
Let $S$ be a finite monoid and $\Bbbk[S]$ the corresponding semigroup algebra which consists of 
all elements of the form $\displaystyle\sum_{s\in S}c_s s$, where $c_s\in\Bbbk$. This algebra
is positively based with
respect to the {\em standard basis} $\mathbf{B}=\{s:s\in S\}$. In fact, all structure constants
with respect to this basis are either zero or one.

\subsection{Example III: Hecke algebras}\label{s2.4}

Let $(W,S)$ be a finite Coxeter system and $\mathcal{H}_v$ the corresponding Hecke algebra over
$\mathbb{Z}[v,v^{-1}]$, in the normalization of \cite{So}. Specializing $v$ to 
\begin{displaymath}
z\in \mathbb{R}_{>0}\bigcup \{u\in\mathbb{C}\,:\,|u|=1\text{ and }\Re(u)>0 \}, 
\end{displaymath}
we get the algebra $\mathcal{H}_z$ defined over the subring of $\mathbb{C}$ generated by $\mathbb{Z}$, 
$z$ and $z^{-1}$. Under our assumption on $z$, we have $z+z^{-1}\in\mathbb{R}_{>0}$. This implies that
the algebra $\mathcal{H}_z$ is positively based with respect to the 
{\em Kazhdan-Lusztig basis}  $\{\underline{H}_w\,:\, w\in W\}$, 
as defined in \cite{KL,So}, see also \cite{EW}. A special case of this construction is the
group algebra $\mathbb{Z}[W]$ of the Coxeter group $W$ (which corresponds to the case $z=1$).

\subsection{Example IV: decategorifications of finitary $2$-categories}\label{s2.5}

The previous example is a special case of the following abstract situation. Let $\cC$ be a finitary 
$2$-category in the sense of \cite{MM1}. Consider its decategorification $[\cC]$ defined via
split Grothendieck groups of the morphism categories, see \cite[Subsection~2.4]{MM2}. Let $A_{\ccC}$ be the
$\mathbb{Z}$-algebra of paths in the category $[\cC]$, defined as
\begin{displaymath}
A_{\ccC}:=\bigoplus_{\mathtt{i},\mathtt{j}\in\ccC}[\cC](\mathtt{i},\mathtt{j}), 
\end{displaymath}
with multiplication naturally induced from composition in $[\cC]$. Then $A_{\ccC}$ is positively based with 
respect to the basis given by isomorphism classes of indecomposable $1$-morphisms in $\cC$.

The example in Subsection~\ref{s2.4} is obtained as a special case if one considers the finitary $2$-category of
{\em Soergel bimodules} (over the coinvariant algebra of $W$) associated to $(W,S)$, see 
\cite[Example~3]{MM2}, \cite[Subsection~6.4]{MM4} and \cite[Subsection~7.3]{MM5}, see also \cite{EW}.

\subsection{Positively based algebras and multistructures}\label{s2.6}

Consider the semiring $(\mathbb{Z}_{>0},+,\cdot,0,1)$ of non-negative integers with respect to 
the usual addition and multiplication. For a positive integer $n$, consider the free module 
$\mathbb{Z}_{>0}^n$ over $\mathbb{Z}_{>0}$ of rank $n$. An {\em $\mathbb{Z}_{>0}$-algebra structure} on $\mathbb{Z}_{>0}^n$
is a map
\begin{displaymath}
*:\mathbb{Z}_{>0}^n\times \mathbb{Z}_{>0}^n\to \mathbb{Z}_{>0}^n
\end{displaymath}
which is bilinear and associative in the usual sense. Defining a $\mathbb{Z}_{>0}$-algebra structure on
$\mathbb{Z}_{>0}^n$ is equivalent to defining, on the standard basis of $\mathbb{Z}_{>0}^n$,
the structure of a multisemigroup with multiplicities in $\mathbb{Z}_{>0}$,
see \cite{Fo} for details. Extending scalars to $\Bbbk$ we get a positively based algebra with the canonical 
positive basis being the standard basis of $\mathbb{Z}_{>0}^n$.

Conversely, if $A$ is a finite dimensional $\Bbbk$-algebra with a fixed positive basis  $\mathbf{B}$
with respect to which all structure constants are integers,
then the $\mathbb{Z}_{>0}$-linear span of $\mathbf{B}$ is a free $\mathbb{Z}_{>0}$-module of finite rank
with the canonical $\mathbb{Z}_{>0}$-algebra structure induced from multiplication in $A$. Extending 
scalars back to $\Bbbk$ recovers $A$.

\section{Positively based algebras: combinatorics and cell modules}\label{s3}

\subsection{The multisemigroup of $A$}\label{s3.1}

Let $A$ be a positively based algebra with a fixed positive basis 
$\mathbf{B}=\{a_i\,:\,i=1,2,\dots,n\}$ as defined in Subsection~\ref{s2.1}.
For simplicity, we will always assume that $a_1$ is the unit element in $A$.
For $i,j\in\{1,2,\dots,n\}$, set
\begin{displaymath}
i\star j:=\{k\,:\, \gamma_{i,j}^{(k)}>0\}. 
\end{displaymath}
This defines an associative multivalued operation on the set $\mathbf{n}:=\{1,2,\dots,n\}$
and thus turns the latter set into a finite {\em multisemigroup}, see \cite[Subsection~3.7]{KuMa}.

\subsection{Cells in $(\mathbf{n},\star)$}\label{s3.2}

For $i,j\in \mathbf{n}$, we set $i\leq_L j$ provided that there is an $s\in \mathbf{n}$ such that 
$j\in s\star i$. Then $\leq_L$ is a partial preorder on $\mathbf{n}$, called the {\em left preorder}.
Write $i\sim_L j$ provided that $i\leq_L j$ and $j\leq_L i$. Then $\sim_L$ is an equivalence relation
on $\mathbf{n}$. Equivalence classes for $\sim_L$ are called {\em left cells}. The preorder
$\leq_L$ induces a genuine partial order on the set of all left cells in $\mathbf{n}$
(which we denote also by $\leq_L$, abusing notation).

Similarly one defines the {\em right preorder} $\leq_R$, the corresponding equivalence relation $\sim_R$
and {\em right cells}, using multiplication with $s$ on the right. Furthermore, one defines the 
{\em two-sided preorder} $\leq_J$, the corresponding equivalence relation $\sim_J$
and {\em two-sided cells}, using multiplication with $s_1$ on the left and $s_2$ on the right.
We write $i<_Lj$ provided that $i\leq_L j$ and $i\not\sim_L j$, and similarly for $i<_R j$ and $i<_J j$.

A two-sided cell $\mathcal{J}$ is said to be {\em idempotent} provided that there exist 
$i,j,k\in\mathcal{J}$ such that $k\in i\star j$.

\subsection{Cell modules}\label{s3.3}

Let $\mathcal{L}$ be a left cell in $\mathbf{n}$ and $\overline{\mathcal{L}}$ be the union of all left cells
$\mathcal{L}'$ in $\mathbf{n}$ such that $\mathcal{L}'\geq \mathcal{L}$. Set 
$\underline{\overline{\mathcal{L}}}:=\overline{\mathcal{L}}\setminus \mathcal{L}$.
Consider the regular $A$-module ${}_AA$ and the $\Bbbk$-submodule $M_{\mathcal{L}}$ of ${}_AA$
spanned by all $a_j$, where $j\in \overline{\mathcal{L}}$. Further, consider the 
$\Bbbk$-submodule $N_{\mathcal{L}}$ of $M_{\mathcal{L}}$
spanned by all $a_j$, where $j\in \underline{\overline{\mathcal{L}}}$.

\begin{proposition}\label{prop1}
Both $M_{\mathcal{L}}$ and $N_{\mathcal{L}}$ are $A$-submodules of ${}_AA$.
\end{proposition}

\begin{proof}
We prove that  $M_{\mathcal{L}}$ is an $A$-submodule of ${}_AA$. That $N_{\mathcal{L}}$ is an $A$-submodule of ${}_AA$
is proved similarly. We need to check that $M_{\mathcal{L}}$ is closed with respect to the left multiplication 
with all $a_i$, where $i\in\mathbf{n}$. For any such $i$ and any $j\in \overline{\mathcal{L}}$, consider
the product $a_i\cdot a_j$ as given by \eqref{eq1}. Note that, for $k\in\mathbf{n}$, our definitions imply 
$\gamma_{i,j}^{(k)}\neq 0$ if and only if $k\geq_L j$. Therefore $a_i\cdot a_j$ is a $\Bbbk$-linear combination 
of the $a_k$'s, for $k\in \overline{\mathcal{L}}$. The claim follows.
\end{proof}

As $N_{\mathcal{L}}\subset M_{\mathcal{L}}$, Proposition~\ref{prop1} allows us to define the
{\em cell $A$-module} $C_{\mathcal{L}}$ associated to $\mathcal{L}$ as the quotient
$M_{\mathcal{L}}/N_{\mathcal{L}}$. Directly from the definitions we have that the regular representation
${}_AA$ has a filtration whose subquotients are isomorphic to cell modules, with each cell module 
occurring at least once, up to isomorphism. 

\subsection{Example I: group algebras}\label{s3.4}

For $A=\Bbbk[G]$, where $G$ is a finite group, with the standard positive basis as described in 
Subsection~\ref{s2.2}, we have the equalities 
$\leq_L=\leq_R=\leq_J=\sim_L=\sim_R=\sim_J=\mathbf{n}\times \mathbf{n}$.
In this case, for the unique left cell $\mathcal{L}=\mathbf{n}$, we have 
$C_{\mathcal{L}}={}_AA$.

\subsection{Example II: semigroup algebras}\label{s3.5}

For $A=\Bbbk[S]$, where $S$ is a finite monoid, with the standard positive basis as described in 
Subsection~\ref{s2.3}, the relations $\sim_L$, $\sim_R$ and $\sim_J$ are exactly the corresponding 
{\em  Green's equivalence relations} as defined in \cite{Gr}. The preorders
$\leq_L$, $\leq_R$, and $\leq_J$ are the corresponding preorders. For a left cell 
$\mathcal{L}$, the corresponding cell module $C_{\mathcal{L}}$ is the usual module
associated with $\mathcal{L}$, see, for example, \cite[Subsection~11.2]{GM}.

\subsection{Example III: Hecke algebras}\label{s3.6}

For $A=\mathcal{H}_v$, the Hecke algebra of a finite Coxeter system $(W,S)$ as in 
Subsection~\ref{s2.4}, with respect to the Kazhdan-Lusztig basis, 
the preorders $\leq_L$, $\leq_R$, and $\leq_J$ are exactly the Kazhdan-Lusztig preorders
and equivalence classes for $\sim_L$, $\sim_R$ and $\sim_J$ are exactly the 
corresponding Kazhdan-Lusztig cells.
The cell module $C_{\mathcal{L}}$ is the Kazhdan-Lusztig cell module, see \cite{KL}.

\subsection{Example IV: decategorifications of finitary $2$-categories}\label{s3.7}

For $A=A_{\ccC}$, where $\cC$ is a finitary $2$-category, as described in 
Subsection~\ref{s2.5}, with respect to the positive basis of indecomposable $1$-morphisms, 
the relations $\leq_L$, $\leq_R$, $\leq_J$, $\sim_L$, $\sim_R$ and $\sim_J$ are described
in \cite{MM1}. The cell module $C_{\mathcal{L}}$ is the decategorification of the 
cell $2$-representation $\mathbf{C}_{\mathcal{L}}$ of $\cC$ defined in \cite{MM1,MM2}.

\section{Perron-Frobenius Theorem}\label{s4}

In this section we recall the following theorem, due to Perron and Frobenius, see \cite{Fr1,Fr2,Pe}. 
It will be a crucial tool in the definition of special modules in the next section.

\begin{theorem}[Perron-Frobenius]\label{thm2}
Let $M\in\mathrm{Mat}_{k \times k}(\mathbb{R}_{>0})$. Then there is a positive real number $\lambda$,
called the {\em Perron-Frobenius eigenvalue} of $M$, such that the following statements hold:
\begin{enumerate}[$($i$)$]
\item\label{thm2.1} The number $\lambda$ is an eigenvalue of $M$.
\item\label{thm2.2} Any other eigenvalue $\mu\in\mathbb{C}$ of $M$ satisfies $|\mu|<\lambda$.
\item\label{thm2.3} The eigenvalue $\lambda$ has algebraic (and hence also geometric) multiplicity $1$.
\item\label{thm2.4} There is $v\in \mathbb{R}_{>0}^k$ such that $Mv=\lambda v$.
There is also $\hat{v}\in \mathbb{R}_{>0}^k$ such that $\hat{v}^tM=\lambda \hat{v}^t$.
\item\label{thm2.5} Any $w\in \mathbb{R}_{\geq 0}^k$ which is an eigenvector of $M$ (with some eigenvalue)
is a scalar multiple of $v$, and similarly for $\hat{v}$.
\item\label{thm2.6} If $v$ and $\hat{v}$ above are chosen such that $\hat{v}^tv=(1)$, then
\begin{displaymath}
\lim_{n\to\infty}\frac{M^n}{\lambda^n}=v\hat{v}^t. 
\end{displaymath}
\end{enumerate}
\end{theorem}

The vector $v\in \mathbb{R}_{>0}^k$ from Theorem~\ref{thm2}\eqref{thm2.4} is called a 
{\em Perron-Frobenius eigenvector}. By Theorem~\ref{thm2}\eqref{thm2.5}, a Perron-Frobenius eigenvector
is defined uniquely up to a positive scalar.

\section{Special modules: definition and basic properties}\label{s5}

\subsection{Perron-Frobenius elements for based modules and special subquotients}\label{s5.1}

Let $\Bbbk$ be a subfield of $\mathbb{C}$. Consider a finite dimensional $\Bbbk$-algebra $A$
and a finite dimensional $A$-module $V$ with a fixed basis $\mathbf{v}=\{v_1,v_2,\dots,v_m\}$.
We will call the pair $(V,\mathbf{v})$ a {\em based $A$-module}. An element $a\in A$ is called 
a {\em Perron-Frobenius} element for a based $A$-module $(V,\mathbf{v})$ provided that all 
entries of the matrix of the action of $a$ on $V$ with respect to the basis $\mathbf{v}$  are 
positive real numbers. 

Given a Perron-Frobenius element $a\in A$ for a based $A$-module $(V,\mathbf{v})$, let $\lambda$
be the Perron-Frobenius eigenvalue of the linear operator $a$ on $V$. A simple $A$-subquotient
$L$ of $V$ is called a {\em special subquotient} with respect to  $a$, provided that $\lambda$
is an eigenvalue of $a$ acting on $L$.  As a consequence of the Perron-Frobenius Theorem, we record
the following:

\begin{corollary}\label{cor3}
Given a Perron-Frobenius element $a\in A$ for a based $A$-module $(V,\mathbf{v})$, there is a unique,
up to isomorphism, special subquotient $L$ of $V$ with respect to $a$, moreover, $[V:L]=1$.
\end{corollary}

\subsection{Perron-Frobenius elements for cell modules}\label{s5.2}

Let $\Bbbk$ be a subfield of $\mathbb{C}$ and $A$ a $\Bbbk$-algebra (of finite dimension $n$ over $\Bbbk$) 
with a fixed positive basis $\mathbf{B}$. For a left cell $\mathcal{L}$ in $\mathbf{n}$, consider
the corresponding cell module $C_{\mathcal{L}}$ as defined in Subsection~\ref{s3.3}. 
Denote by $\mathbf{B}_{\mathcal{L}}$ the {\em standard basis} of $C_{\mathcal{L}}$ given 
by the images of the elements $a_i$, where $i\in \mathcal{L}$.

For $i=1,2,\dots,n$, fix some positive real numbers $c_i\in \Bbbk$. Set $\mathbf{c}:=(c_1,c_2,\dots,c_n)$ and
\begin{equation}\label{eq9}
a(\mathbf{c}):=\sum_{i=1}^n c_i a_i\in A.
\end{equation}

\begin{lemma}\label{lem4}
The element $a(\mathbf{c})$ is a Perron-Frobenius element for the based module 
$(C_{\mathcal{L}},\mathbf{B}_{\mathcal{L}})$.
\end{lemma}

\begin{proof}
Since $\mathbf{B}$ is a positive basis, it follows that all entries of the matrix of the action 
of each $a_i$, where $i\in\mathbf{n}$, on $C_{\mathcal{L}}$ in the basis $\mathbf{B}_{\mathcal{L}}$ 
are non-negative real numbers. 

Let $i,j\in\mathcal{L}$. Then there is $k\in\mathbf{n}$ such that $\gamma_{k,j}^{(i)}\neq 0$, which
implies that the $(i,j)$-th entry in the matrix of the action of $a_k$ on $C_{\mathcal{L}}$ is positive.
As $c_k>0$, combined with the previous paragraph, we get that the $(i,j)$-th entry in the matrix 
of the action of $a(\mathbf{c})$ on $C_{\mathcal{L}}$ is positive. As $i$ and $j$ were arbitrary,
the claim follows.
\end{proof}

\subsection{Special subquotients of cell modules}\label{s5.3}

For each left cell $\mathcal{L}$ and each $\mathbf{c}\in(\mathbb{R}_{>0}\cap \Bbbk)^n$, the discussion
above allows us to define the corresponding {\em special} subquotient $L_{\mathcal{L},\mathbf{c}}$ of
$C_{\mathcal{L}}$.

\subsection{Independence of $\mathbf{c}$}\label{s5.4}

Our first main observation is the following:

\begin{theorem}\label{thm5}
For a fixed left cell $\mathcal{L}$ and any $\mathbf{c},\mathbf{c}'\in(\mathbb{R}_{>0}\cap \Bbbk)^n$, 
there is an isomorphism $L_{\mathcal{L},\mathbf{c}}\cong L_{\mathcal{L},\mathbf{c}'}$.
\end{theorem}

\begin{proof}
Assume first that $\Bbbk=\mathbb{C}$. Consider the map $\lambda:\mathbb{R}_{>0}^n\to \mathbb{R}_{>0}$
which sends $\mathbf{c}$ to the Perron-Frobenius eigenvalue of $a(\mathbf{c})$ on $C_{\mathcal{L}}$.
This map is, obviously, continuous. Let $\mathrm{Irr}(A)$ be the set of isomorphism classes of simple $A$-modules.
Consider the map $L_{\mathcal{L},{}_-}:\mathbb{R}_{>0}^n\to \mathrm{Irr}(A)$ which sends $\mathbf{c}$ to
$L_{\mathcal{L},\mathbf{c}}$. For $L\in \mathrm{Irr}(A)$, consider its preimage $X_L$ under the latter map
and assume it is non-empty.

We claim that from Theorem~\ref{thm2} it follows that $X_L$ is closed in $\mathbb{R}_{>0}^n$. Indeed, 
let $\mathbf{c}_i$ be a sequences in $\mathbb{R}_{>0}^n\cap X_L$ which converges to $\mathbf{c}\in \mathbb{R}_{>0}^n$.
Let $L_1,L_2,\dots,L_k=L$ be the list of simple subquotients of $C_{\mathcal{L}}$. By Theorem~\ref{thm2},
for each $i$, the value  $\lambda(\mathbf{c}_i)$ is the maximal absolute value of an
eigenvalue of $a(\mathbf{c}_i)$ on $L_k$ and is strictly bigger than the absolute value of any other 
eigenvalue of $a(\mathbf{c}_i)$ on any of the $L_j$'s. Taking the limit, we get that the maximal
absolute value of an eigenvalue of $a(\mathbf{c})$ on $L_k$ is not less than the supremum of the
absolute values over all other  eigenvalues of $a(\mathbf{c})$ on any of the $L_j$'s. Since 
the Perron-Frobenius eigenvalue has multiplicity one, it follows that $L_{\mathcal{L},\mathbf{c}}$ is still
isomorphic to $L$.

By  noting that $\mathrm{Irr}(A)$ is finite with discrete topology, 
we see that the preimage of a closed set under $L_{\mathcal{L},{}_-}$ is closed.
Therefore $L_{\mathcal{L},\mathbf{c}}$ is continuous and thus must be constant since
$\mathbb{R}_{>0}^n$ is connected.

If $\Bbbk\neq \mathbb{C}$ and the assertion of the theorem fails, then we can extend scalars  
from $\Bbbk$ to $\mathbb{C}$ and, because of the multiplicity one property established in Corollary~\ref{cor3}, 
obtain that the assertion of the theorem must also fail for 
the case $\Bbbk=\mathbb{C}$, which contradicts the above. This completes the proof.
\end{proof}

As $L_{\mathcal{L},\mathbf{c}}$ does not really depend on $\mathbf{c}$ by Theorem~\ref{thm5}, we 
will denote this module by  $L_{\mathcal{L}}$. In this way, we define a map from the set of left cells in
$\mathbf{n}$ to the set $\mathrm{Irr}(A)$ of isomorphism classes of simple $A$-modules. In general,
this map is neither injective nor surjective. 

\subsection{$J$-invariance of special subquotients}\label{s5.5}

Our second main observation is the following:

\begin{theorem}\label{thm7}
For any two left cells $\mathcal{L}$ and $\mathcal{L}'$ which belong to the same two-sided cell, 
we have $L_{\mathcal{L}}\cong L_{\mathcal{L}'}$.
\end{theorem}

\begin{proof}
Denote by $\mathcal{J}$ the two-sided cell containing both $\mathcal{L}$ and $\mathcal{L}'$.
Without loss of generality, we may assume that the cell $\mathcal{L}'$ is minimal, with respect to $\leq_L$,
in the set of all left cells contained in $\mathcal{J}$. Consider the element
$a=a_1+a_2+\dots+a_n\in A$ and the cell modules $C_{\mathcal{L}}$ and $C_{\mathcal{L}'}$.

\begin{lemma}\label{lthm7}
For any $i\in\mathcal{L}$,
the set $i\star \mathbf{n}$ intersects all left cells which are minimal, 
with respect to $\leq_L$, in the set of all left cells contained in $\mathcal{J}$.
\end{lemma}

\begin{proof}
Since $\mathcal{J}$ is a two-sided cell, the set $\mathbf{n}\star i\star \mathbf{n}$ contains $\mathcal{J}$.
For any $j\in i\star \mathbf{n}$, any element $s\in \mathbf{n}\star j$ satisfies $s\geq_L j$. This
implies the claim of the lemma.
\end{proof}

From Lemma~\ref{lthm7}, we have that there
is $j\in \mathbf{n}$ such that $i\star j$ contains some element in $\mathcal{L}'$.
Now, right multiplication with $a_j$ followed by the projection onto $C_{\mathcal{L}'}$, defines
an $A$-module homomorphism  $\varphi$ from $C_{\mathcal{L}}$ to $C_{\mathcal{L}'}$. 
This homomorphism is non-zero by our choice of $j$ and it
sends, by construction, any linear combination of elements in $\mathbf{B}_{\mathcal{L}}$ with positive 
coefficients to a non-zero element in $C_{\mathcal{L}'}$. 

Let $v\in C_{\mathcal{L}}$ be an eigenvector of $a$ which is a linear combination of 
elements in $\mathbf{B}_{\mathcal{L}}$ with positive  coefficients. Note that $v$ exists by
Theorem~\ref{thm2}\eqref{thm2.4} and is unique up to a positive scalar by Theorem~\ref{thm2}\eqref{thm2.5}. 
Then $\varphi(v)$ is a non-zero eigenvector of $a$ in $C_{\mathcal{L}'}$. Since $v$ determines 
the subquotient $L_{\mathcal{L}}$ uniquely, it follows that this subquotient is not annihilated 
by $\varphi$. On the other hand, by construction, $\varphi(v)$ is a non-zero linear combination 
of elements in $\mathbf{B}_{\mathcal{L}'}$ with non-negative coefficients. Therefore the corresponding
eigenvalue is the Perron-Frobenius eigenvalue of $a$ for $C_{\mathcal{L}'}$ by 
Theorem~\ref{thm2}\eqref{thm2.5}. Combined with the definition of special subquotient, it follows
that $\varphi(L_{\mathcal{L}})\cong L_{\mathcal{L}'}$, completing the proof.
\end{proof}

\section{Special subquotients of cell modules: examples}\label{s6}

\subsection{Group algebras}\label{s6.1}

Let $G$ be a finite group and $A=\Bbbk[G]$ the corresponding group algebra.
Then we have the unique left cell $\mathcal{L}=\mathbf{n}$ and the corresponding 
cell module is just the left regular module ${}_AA$. The element
\begin{displaymath}
a:=\sum_{g\in G}g,
\end{displaymath}
considered as an element of the algebra $A$, is a Perron-Frobenius element for the
module ${}_AA$. On the other hand, the same element can be considered as an 
element of  $C_{\mathcal{L}}$ and we have $a\cdot a=|G|a$. Therefore, by
Theorem~\ref{thm2}\eqref{thm2.5}, the value $|G|$ is the
Perron-Frobenius eigenvalue of $A$ on ${}_AA$ and hence the special subquotient of 
${}_AA$ is the trivial $A$-module (represented inside ${}_AA$ as the submodule $\Bbbk a$).

The above admits the following generalization. Let $H$ be a subgroup of $G$ and
$\Bbbk[G/H]\cong\mathrm{Ind}^G_H(\mathrm{triv}_H)$ 
be the corresponding permutation module given by the left action of $G$ on the set of all cosets $gH$, where 
$g\in G$. The action of $G$ on the set of all such cosets is transitive and hence $a$ is a 
Perron-Frobenius element for the module $\Bbbk[G/H]$ with respect to the canonical basis given by the cosets.
The sum of all basis elements spans a submodule isomorphic to the trivial $G$-module and is 
an eigenvector for $a$. Therefore the special subquotient of $\Bbbk[G/H]$ is again
the trivial $A$-module.

\subsection{Semigroup algebras}\label{s6.2}

Let $S$ be a finite monoid with $|S|=n$ and fix the standard positive basis 
$\mathbf{B}=\{s\,:\,s\in S\}$ in the semigroup algebra $\Bbbk[S]$. In this setup, 
cells in $\mathbf{n}$ correspond to Green's equivalence relations on $S$,
see \cite{Gr} or \cite[Chapter~4]{GM}. Let $\mathcal{L}$ be a left cell and
$\mathcal{J}$ be the apex of $C_{\mathcal{L}}$ as defined in Section~\ref{s8}. Then $\mathcal{J}$ is a regular
$J$-class (see Proposition~\ref{prop1-1} below) and hence contains an 
idempotent, say $e$, see also \cite{GMS}.  Let $\mathcal{L}_e$ be the left cell containing $e$.
Let $G$ be the maximal subgroup of $S$ which corresponds to $e$. 
Right multiplication with elements of $G$ induces on the
$\Bbbk[S]$-module $C_{\mathcal{L}_e}$ the structure of a
$\Bbbk[S]$-$\Bbbk[G]$-bimodule. Let $\Bbbk_{\mathrm{triv}}$ denote the trivial
$G$-module, that is the vector space $\Bbbk$ on which all elements of $G$ act
as the identity operator. Then, by adjunction, the $S$-module
\begin{displaymath}
\Delta(\mathcal{L}_e,\Bbbk_{\mathrm{triv}}):=C_{\mathcal{L}_e}\bigotimes_{\Bbbk[G]}\Bbbk_{\mathrm{triv}} 
\end{displaymath}
has simple top which we denote by $L_e$, see \cite{GMS} or \cite[Chapter~11]{GM}.

\begin{proposition}\label{prop21}
The simple $\Bbbk[S]$-module $L_e$ is the special subquotient of $C_{\mathcal{L}}$.
\end{proposition}

\begin{proof}
Denote by $L$ the special subquotient of $C_{\mathcal{L}}$. Take any 
\begin{displaymath}
a= \sum_{s\in S}a_s s, \quad\text{ where } a_s\in\mathbb{R}_{>0}\cap \Bbbk,
\end{displaymath}
and let $v\in C_{\mathcal{L}}$ be a corresponding Perron-Frobenius eigenvector.
Consider the element
\begin{displaymath}
b=\sum_{g\in G}g\in A. 
\end{displaymath}
Because of our definition of $G$, the element $bv$ is a non-zero linear combination
of elements in $\mathbf{B}_{\mathcal{L}}$ with non-negative real coefficients.
From Proposition~\ref{prop51} we thus obtain that the image of $bv$ in $L$ is non-zero.
Therefore $bL\neq 0$.

Note that $xb=b$, for any $x\in G$, by construction. Therefore  $xbv=bv$, for all $x\in G$.
This means that $\mathrm{Res}^S_G(L)$ contains a submodule isomorphic to $\Bbbk_{\mathrm{triv}}$.
By adjunction, it follows that $\Delta(\mathcal{L}_e,\Bbbk_{\mathrm{triv}})$ surjects onto
$L$. This implies $L_e=L$ and completes the proof.
\end{proof}

\subsection{Group algebras of finite Weyl groups}\label{s6.3}

For a finite Weyl group $W$, consider the group algebra $A:=\mathbb{C}[W]$. It is positively based 
with respect to the Kazhdan-Lusztig basis $\mathbf{B}:=\{\underline{H}_w\,:\, w\in W\}$
(we follow the normalization of \cite{So}). Left
cell representations of $A$ in this setup are exactly the Kazhdan-Lusztig left cell modules. By
\cite{Lu3} (see also \cite{Ge1,Ge2} for alternative approaches and further details), 
the class of Kazhdan-Lusztig left cell modules coincides with
the class of {\em cells} or {\em constructible representations} as defined in \cite{Lu2}.
In \cite{Lu1}, Lusztig defines in this setup the class of so-called {\em special} representations
and in \cite{Lu2} shows that each constructible representation has exactly one special subquotient
(occurring with multiplicity one). 

\begin{proposition}\label{prop41}
Let $\mathcal{L}$ be a left cell in $A$.
The $L_{\mathcal{L}}$ is a special representation in the sense of Lusztig.
\end{proposition}

\begin{proof}
It is enough to consider the case of irreducible $W$. If $W$ is of type $G_2$,
the assertion is proved by a direct calculation. If $W$ is of type $F_4$, then from the lists in
\cite{Lu2} it follows that, for each two-sided cell $\mathcal{J}$ in $W$, the special 
representation in the sense of Lusztig is the only simple representation which appears with
multiplicity one in all constructible representations associated to $\mathcal{J}$.
This implies the claim of the proposition for type $F_4$.
Therefore, from now on, we assume that $W$ is not of type $G_2$ or $F_4$.

Let $\mathcal{J}$ be the two-sided cell containing $\mathcal{L}$ and $\mathcal{L}=\mathcal{L}_1$,
$\mathcal{L}_2$,\dots, $\mathcal{L}_k$ be a complete list of all left cells in $\mathcal{J}$.
By Theorem~\ref{thm7}, we know that $L_{\mathcal{L}}=L_{\mathcal{L}_i}$ for all $i=1,2,\dots,k$.
Therefore the assertion of our proposition follows directly from the following lemma:

\begin{lemma}\label{lem42}
Assume $W$ is irreducible and not of type $G_2$ or $F_4$. Let $\mathcal{J}$ be a 
two-sided cell in $W$ and $\mathcal{L}_1$, $\mathcal{L}_2$,\dots, $\mathcal{L}_k$ be a complete 
list of all left cells in $\mathcal{J}$. Then the $A$-modules $C_{\mathcal{L}_1}$,
$C_{\mathcal{L}_2}$,\dots, $C_{\mathcal{L}_k}$ have exactly one simple subquotient in common.
\end{lemma}

We note that Lemma~\ref{lem42} fails if $W$ is of type $G_2$ or $F_4$, as follows easily from the lists 
of constructible characters given in \cite{Lu2}.

\begin{proof}
For all exceptional  types the assertion follows from the lists of constructible characters given in \cite{Lu2}.
In type $A$, all $L_{\mathcal{L}_i}$ are simple and isomorphic (see \cite{Lu2}) and hence the assertion
is obvious. It remains to consider the types $B$ and $D$. In both cases the assertion follows from the
description of constructible representations given in \cite{Lu2}. We outline the argument for type $B$
and leave type $D$ as an exercise for the reader.

It type $B_n$, following \cite{Lu0}, irreducible representations are indexed by
certain (equivalence classes of) arrays of the form
\begin{displaymath}
\left(\begin{array}{c}\boldsymbol{\lambda}\\\boldsymbol{\mu}\end{array}\right)=
\left(\begin{array}{ccccc}\lambda_1&\lambda_2&\dots&\lambda_{m}&\lambda_{m+1}\\
\mu_1&\mu_2&\dots&\mu_m&\end{array}\right), 
\end{displaymath}
called {\em symbols}. These symbols have, among others, the following properties:
\begin{itemize}
\item all entries are non-negative integers which add up to $n+m^2$ for some $m$;
\item each integer appears in the array at most twice and, in case some integer appears twice, 
it appears both in the first and in the second rows;
\item elements in both rows increase from left to right.
\end{itemize}
Such an array indexes a special representation in the sense of \cite{Lu1} if and only if 
\begin{equation}\label{eq11}
\lambda_1\leq \mu_1\leq\lambda_2\leq\mu_2\leq\dots\leq\lambda_{m+1}.
\end{equation}
We will say that such a symbol is {\em special}.

Let us fix a special symbol as above and let 
\begin{equation}\label{eq12}
\gamma_1<\gamma_2<\gamma_3<\dots<\gamma_{2k+1} 
\end{equation}
be the sequence obtained from \eqref{eq11} be deleting all elements which appear twice.
Constructible representations containing the special representation corresponding to the 
special symbol above are indexed by certain partitions of elements in \eqref{eq12} in pairs 
which leave one singleton, as described in \cite{Lu2}. Each pair contains exactly one entry 
in the first row and  exactly one entry in the second row of our special symbol. All other 
subquotients of the corresponding constructible representation are obtained by swapping entries in 
some of these pairs. This gives exactly $2^k$ subquotients. 

Consider the two constructible representations corresponding to the partitions
\begin{displaymath}
\{\gamma_1,\gamma_2\}\quad
\{\gamma_3,\gamma_4\}\quad\dots\quad
\{\gamma_{2k-1},\gamma_{2k}\}\quad \{\gamma_{2k+1}\}
\end{displaymath}
and
\begin{displaymath}
\{\gamma_1\}\quad 
\{\gamma_2,\gamma_3\}\quad\{\gamma_4,\gamma_5\}\quad\dots\quad
\{\gamma_{2k},\gamma_{2k+1}\}.
\end{displaymath}
From the above it is straightforward that the special representation is the only common
composition subquotient of these two constructible representations. This completes the proof
in type $B_n$.
\end{proof}

This completes the proof of Proposition~\ref{prop41}.
\end{proof}

\section{Special modules: further properties}\label{s7}

\subsection{Projection of the positive cone}\label{s7.1}

Let $\mathcal{L}$ be a left cell in $(\mathbf{n},\star)$. Consider the modules $C_{\mathcal{L}}$ and
$L_{\mathcal{L}}$. Let $P_{\mathcal{L}}$ denote the indecomposable projective cover of $L_{\mathcal{L}}$
in $A$-mod. Then $\dim \mathrm{Hom}_A(P_{\mathcal{L}},C_{\mathcal{L}})=1$ and we can denote by 
$V_{\mathcal{L}}$ the image in $C_{\mathcal{L}}$ of any non-zero homomorphism from $P_{\mathcal{L}}$.
The $A$-module $V_{\mathcal{L}}$ has simple top $L_{\mathcal{L}}$ by construction, and we denote by 
$K_{\mathcal{L}}$ the kernel of the canonical projection $V_{\mathcal{L}}\tto L_{\mathcal{L}}$.
Note that both $V_{\mathcal{L}}$ and $K_{\mathcal{L}}$ are submodules of $C_{\mathcal{L}}$.

For $\mathbf{c}\in(\mathbb{R}_{>0}\cap \Bbbk)^n$, consider the corresponding element $a(\mathbf{c})\in A$
as defined in \eqref{eq9}. Let $\lambda(\mathbf{c})$ denote the Perron-Frobenius eigenvalue of
$a(\mathbf{c})$ on $C_{\mathcal{L}}$. Let, further,  $v(\mathbf{c})\in C_{\mathcal{L}}$ be 
a Perron-Frobenius eigenvector for $a(\mathbf{c})$.

\begin{lemma}\label{lem52}
We have $V_{\mathcal{L}}=Av(\mathbf{c})$, in particular, the submodule $Av(\mathbf{c})$ of $C_{\mathcal{L}}$
does not depend on the choice of $\mathbf{c}$.
\end{lemma}

\begin{proof}
On the one hand, we have $v(\mathbf{c})\in V_{\mathcal{L}}$ by construction and hence
$Av(\mathbf{c})\subset V_{\mathcal{L}}$. On the other hand, $V_{\mathcal{L}}$ has simple top 
$L_{\mathcal{L}}$ and the latter simple has multiplicity one in $V_{\mathcal{L}}$. By construction, 
$L_{\mathcal{L}}$ also has a non-zero multiplicity in $Av(\mathbf{c})$. Therefore
$V_{\mathcal{L}}=Av(\mathbf{c})$ and the claim follows.
\end{proof}

Denote by $\mathbf{B}_{\mathcal{L}}^+$ the subset of $C_{\mathcal{L}}$ consisting of all possible
linear combinations of elements in $\mathbf{B}_{\mathcal{L}}$ with {\em non-negative} coefficients.

\begin{proposition}\label{prop51}
We have $K_{\mathcal{L}}\cap \mathbf{B}_{\mathcal{L}}^+=0$.
\end{proposition}

\begin{proof}
We assume $\Bbbk=\mathbb{R}$, all other cases follow from this one by restricting and extending scalars. 
We have  $K_{\mathcal{L}}\cap \mathbf{B}_{\mathcal{L}}^+\neq \mathbf{B}_{\mathcal{L}}^+$ since any
$v(\mathbf{c})$ as above is in $\mathbf{B}_{\mathcal{L}}^+$ and, certainly, is not in
$K_{\mathcal{L}}$. Assume that $X:=K_{\mathcal{L}}\cap \mathbf{B}_{\mathcal{L}}^+$ contains some
non-zero $v$. Then $X$ contains all $\lambda v$, where $\lambda\in\mathbb{R}_{>0}$.

As $K_{\mathcal{L}}$ is a submodule, it follows that $X$ is invariant under the action of any 
$a(\mathbf{c})$ as above. Consider the inner product in $C_{\mathcal{L}}$ for which 
$\mathbf{B}_{\mathcal{L}}$ is an orthonormal basis. Let $X_1$ be the subset of $X$ consisting of
all elements of length one in $X$. Let $X_2$ be the convex hull of $X_1$. Clearly, $X_1$ is compact and
non-empty. Therefore $X_2$ is compact, convex and non-empty. Since $X_1$ does not contain zero,
$X_2$ does not contain zero either, by construction. Consider the  transformation
\begin{equation}\label{eq51}
v\mapsto \frac{a(\mathbf{c})\cdot v}{|a(\mathbf{c})\cdot v|}
\end{equation}
of $X_2$ which is well-defined and continuous as $|a(\mathbf{c})\cdot v|\neq 0$ since $1$ 
appears in  $a(\mathbf{c})$ with a non-zero coefficient. 
By the Schauder fixed point theorem, see \cite{Sch}, the transformation \eqref{eq51} 
must have a fixed point. Any such fixed
point is, by construction, an eigenvector of $a(\mathbf{c})$ lying in $\mathbf{B}_{\mathcal{L}}^+$ 
and different from the  Perron-Frobenius eigenvector (since it is in $K_{\mathcal{L}}$). 
This contradicts uniqueness of 
the Perron-Frobenius eigenvector in Theorem~\ref{thm2}\eqref{thm2.5}. Therefore $X_1=\varnothing$.
\end{proof}

\subsection{Special modules for semi-simple algebras}\label{s7.2}

\begin{proposition}\label{propn37}
Assume that $\Bbbk$ is algebraically closed and $A$ is semi-simple.

\begin{enumerate}[$($i$)$]
\item\label{propn37.1} Each two-sided cell for $A$ is idempotent.
\item\label{propn37.2} Let $\mathcal{L}$ be a left cell and $\mathcal{J}$ a two-sided cell containing
$\mathcal{L}$. Then the dimension of $L_{\mathcal{L}}$ equals the number of left cells in $\mathcal{J}$.
\item\label{propn37.3} Any simple subquotient of $C_{\mathcal{L}}$ different from $L_{\mathcal{L}}$
is not isomorphic to $L_{\mathcal{L}'}$ for any left cell $\mathcal{L}'$.
\end{enumerate}
\end{proposition}

\begin{proof}
Let $\mathcal{J}$ be a two-sided cell in $A$. As all quotients of semi-simple algebras are semi-simple,
by taking an appropriate quotient of $A$ we may assume that $\mathcal{J}$ is maximal with respect to 
$\leq_J$. If $\mathcal{J}$ is not idempotent, then the linear span of all $a_i$, where $i\in \mathcal{J}$,
is a nilpotent ideal of $A$. This contradicts semi-simplicity of $A$ and proves claim~\eqref{propn37.1}.

Fix an ordering $\mathcal{J}_1,\mathcal{J}_2,\dots,\mathcal{J}_k$ of two-sided cells such that
$i<j$ implies $\mathcal{J}_i\not\geq_J\mathcal{J}_j$.
For $i=1,2,\dots,k$, denote by $I_i$ the linear span of all $a_j$, where $a_j\in \mathcal{J}_s$ and $s\leq i$.
Then
\begin{equation}\label{eqn9}
0=I_0\subset I_1\subset \dots\subset I_k=A 
\end{equation}
is a filtration of $A$ by two-sided ideals. As $A$ is semi-simple, each simple $A$-module $L$ appears in
${}_AA$ with multiplicity $\dim(L)$. As \eqref{eqn9} is a filtration by two-sided ideals,
there is $i\in\{1,2,\dots,k\}$ such that 
$L$ appears with multiplicity $\dim(L)$ in $I_i/I_{i-1}$. At the same time, the $A$-module
$I_i/I_{i-1}$ is isomorphic to the direct sum of $C_{\mathcal{L}''}$, where $\mathcal{L}''$ runs through the
set of all left cells in $\mathcal{J}_i$. This implies claim~\eqref{propn37.3}. Claim~\eqref{propn37.2}
now follows from the observation that $L_{\mathcal{L}}$ appears exactly once in each 
$C_{\mathcal{L}'}$ if $\mathcal{L}$ and $\mathcal{L}'$ belong to the same two-sided cell.
\end{proof}

\section{The apex}\label{s8}

\subsection{The apex of a cell module}\label{s8.1}

Given a left cell $\mathcal{L}$, consider the set $\mathcal{X}_{\mathcal{L}}$ of all two-sided cells
$\mathcal{J}$ in $(\mathbf{n},\star)$ for which there exists $i\in \mathcal{J}$ with the property that 
$a_i\cdot C_{\mathcal{L}}\neq 0$. The set $\mathcal{X}_{\mathcal{L}}$ is partially 
ordered with respect to $\leq_J$.

\begin{proposition}\label{prop1-1} 
Let $\mathcal{L}$ be a left cell in $(\mathbf{n},\star)$.
\begin{enumerate}[$($i$)$]
\item\label{prop1-1.1} The set $\mathcal{X}_{\mathcal{L}}$ contains a maximum element,
denoted $\mathcal{J}(\mathcal{L})$.
\item\label{prop1-1.2} The two-sided cell  $\mathcal{J}(\mathcal{L})$ is idempotent.
\item\label{prop1-1.3} For any $i\leq_J \mathcal{J}(\mathcal{L})$, we have $a_i\cdot C_{\mathcal{L}}\neq 0$.
\end{enumerate}
\end{proposition}

\begin{proof}
Let $\mathcal{J}$ be a maximal element in $\mathcal{X}_{\mathcal{L}}$ and $i\leq_J \mathcal{J}$.
We first show the inequality  $a_i\cdot C_{\mathcal{L}}\neq 0$. In particular, given claim~\eqref{prop1-1.1}, this
would imply claim~\eqref{prop1-1.3}. Assume that $a_i\cdot C_{\mathcal{L}}=0$ for some 
$i\leq_J \mathcal{J}$. Then $Aa_iA\cdot C_{\mathcal{L}}=0$. 

As $\mathcal{J}\in\mathcal{X}_{\mathcal{L}}$, there is $j\in \mathcal{J}$ such that 
$a_j\cdot C_{\mathcal{L}}\neq 0$. Let $s,t\in \mathbf{n}$ be such that 
$a_j$ appears in $a_sa_ia_t$ with a non-zero coefficient. Note that the matrix of the action of
each $a_q$, where $q\in\mathbf{n}$, on $C_{\mathcal{L}}$ in the basis $\mathbf{B}_{\mathcal{L}}$
has only non-negative coefficients. Putting this together with the fact that 
$a_sa_ia_t$ is a linear combination of the $a_q$'s with non-negative coefficients, we have that 
$a_sa_ia_t\cdot C_{\mathcal{L}}=0$ implies $a_j\cdot C_{\mathcal{L}}=0$, a contradiction.
This shows that $a_i\cdot C_{\mathcal{L}}\neq 0$.
 
Assume now that $\mathcal{J}$ and $\mathcal{J}'$ are two different maximal elements in $\mathcal{X}_{\mathcal{L}}$.
Let $i\in \mathcal{J}$ and $j\in \mathcal{J}'$. Then the product $a_ia_j$ is a linear combination of 
$a_q$, where $q>_J\mathcal{J}$. Hence $a_ia_j\cdot C_{\mathcal{L}}=0$. On the other hand, let 
$B$ be the subspace of $A$ spanned by all $a_j$, where $j\in \mathcal{J}'$. Then each element of 
$AB$ is a linear combination of $a_q$, where $q\geq _J \mathcal{J}'$. Therefore
$AB\cdot C_{\mathcal{L}}=B\cdot C_{\mathcal{L}}$. Now, from $a_ia_j\cdot C_{\mathcal{L}}=0$
(for all $i\in \mathcal{J}$ and $j\in \mathcal{J}'$), we have
$a_iB\cdot C_{\mathcal{L}}=0$, for all $i\in  \mathcal{J}$. Below we show that this leads to a
contradiction.

Consider some non-zero $v= a_j\cdot a_p\in C_{\mathcal{L}}$, where $j\in \mathcal{J}'$ and $p\in \mathcal{L}$ and
write it as a linear combination of elements in $\mathbf{B}_{\mathcal{L}}$. Assume that some $a_s$ 
appears in this linear combination with a non-zero coefficient. Consider now  $a_t\cdot v$, for all 
$t\in \mathbf{n}$ for which  $a_t\cdot v\neq 0$. Positivity of the basis $\mathbf{B}$ 
and the fact that $\mathcal{L}$ is a left cell imply that, for each $q\in \mathcal{L}$, there
is some $t\in\mathbf{n}$ such that $a_q$ appears with a non-zero coefficient in $a_t\cdot v$.

Take now any $i\in\mathcal{J}$. Then, because of the first claim, there must exist $q\in \mathcal{L}$ such that 
$a_i\cdot a_q\neq 0$ in $C_{\mathcal{L}}$. 
Positivity of the basis $\mathbf{B}$ now implies $a_ia_t\cdot v\neq 0$, a contradiction.
This proves claim~\eqref{prop1-1.1} and hence also claim~\eqref{prop1-1.3} because of the argument above.

Claim~\eqref{prop1-1.2} is proved by a slight modification of the above argument used to prove
claim~\eqref{prop1-1.1}. In more details, in the above argument take $\mathcal{J}=\mathcal{J}'=
\mathcal{J}(\mathcal{L})$  and assume that $a_ia_j\cdot C_{\mathcal{L}}=0$, 
for all $i,j\in \mathcal{J}(\mathcal{L})$. Following the argument all the way through, we get a
contradiction. This completes the proof.
\end{proof}

The two-sided cell $\mathcal{J}(\mathcal{L})$ is called {\em the apex} of $C_{\mathcal{L}}$. 
In the case of semigroups, the notion of the apex of a simple module is standard, see
\cite{Mu,GMS}. In our case, however, we do not know whether one could define a sensible notion of
apex for all simple $A$-modules. Our arguments in Proposition~\ref{prop1-1} rely heavily on the
fact that $C_{\mathcal{L}}$ has a positive basis, that is a basis in which the action of all $a_i$
is given by a matrix with non-negative coefficients.
In the setup of $2$-categories, the notion of apex is discussed in \cite[Section~4]{CM}.

\subsection{$\mathcal{J}$-invariance}\label{s8.2}

\begin{proposition}\label{prop31}
Let $\mathcal{L}$ and  $\mathcal{L}'$ be two left cells in $(\mathbf{n},\star)$ which belong to the
same two-sided cell. Then $\mathcal{J}(\mathcal{L})=\mathcal{J}(\mathcal{L}')$.
\end{proposition}

\begin{proof}
We use the same trick as in the proof of Theorem~\ref{thm7}. Let $\mathcal{I}$ be the two-sided cell
containing both $\mathcal{L}$ and  $\mathcal{L}'$. Without loss of generality, we may 
assume that $\mathcal{L}'$ is minimal, with respect to $\leq_L$, in the set of all left cells
contained in $\mathcal{I}$. Let $\varphi:C_{\mathcal{L}}\to C_{\mathcal{L}'}$ be the homomorphism
constructed in the proof of Theorem~\ref{thm7}. 

Fix some $\mathbf{c}\in\mathbb{R}^n_{>0}$ and let $v_{\mathbf{c}}\in C_{\mathcal{L}}$ and
$v'_{\mathbf{c}}\in C_{\mathcal{L}'}$ be Perron-Frobenius eigenvectors for $a(\mathbf{c})$. We may
assume that $\varphi(v_{\mathbf{c}})=v'_{\mathbf{c}}$. Note that all elements in $\mathbf{B}_{\mathcal{L}}$
appear in the expression of $v_{\mathbf{c}}$ with positive coefficients and similarly for 
$\mathbf{B}_{\mathcal{L}'}$ and $v'_{\mathbf{c}}$. Because of the positivity property of the basis 
$\mathbf{B}$, we see that an element $a_i$ annihilates $C_{\mathcal{L}}$ if and only if it annihilates
$v_{\mathbf{c}}$. Similarly, $a_i$ annihilates $C_{\mathcal{L}'}$ if and only if it annihilates
$v'_{\mathbf{c}}$. This implies $\mathcal{J}(\mathcal{L}')\leq_J\mathcal{J}(\mathcal{L})$.

To prove equality, let $\mathcal{L}_1$ be a left cell in $\mathcal{J}(\mathcal{L})$ which is maximal
with respect to $\leq_L$. Consider the element 
\begin{displaymath}
a=\sum_{i\in \mathcal{L}_1}a_i. 
\end{displaymath}
Let $j\in\mathcal{L}$ be such that $a_ia_j\neq 0$ in $C_{\mathcal{L}}$ 
for some (and hence for all) $i\in \mathcal{L}_1$. Consider the non-zero 
element $aa_j\in C_{\mathcal{L}}$. The latter element is a linear combination 
of elements in $\mathbf{B}_{\mathcal{L}}$ with 
non-negative coefficients. Let $X\subset \mathcal{L}$ be the set of all indexes for which the corresponding basis
vectors appear in $aa_j$ with positive coefficients. Because of the maximality of 
$\mathcal{L}_1$  with respect to $\leq_L$, the linear combination of all
$a_x$, where $x\in X$, is $A$-invariant.  Now, the fact that $\mathcal{L}$ is a left cell, 
implies $X=\mathcal{L}$. Consequently, because of the positivity 
property of the basis  $\mathbf{B}$, we have $\varphi(av_{\mathbf{c}})\neq 0$.
This means that $a$ does not annihilate $C_{\mathcal{L}'}$ and thus implies 
$\mathcal{J}(\mathcal{L}')=\mathcal{J}(\mathcal{L})$.
\end{proof}

Because of Proposition~\ref{prop31}, we may define the {\em apex} $\mathcal{J}(\mathcal{I})$,
for any two-sided cell $\mathcal{I}$, via $\mathcal{J}(\mathcal{I}):=\mathcal{J}(\mathcal{L})$,
where $\mathcal{L}$ is a left cell in $\mathcal{I}$.

\subsection{The apex and special modules}\label{s8.3}

\begin{corollary}\label{cor32}
Let $\mathcal{L}$ be a left cell in $(\mathbf{n},\star)$. Then, for $i\in\mathbf{n}$, the element
$a_i$ does not annihilate $L_{\mathcal{L}}$ if and only if $i\leq_J \mathcal{J}(\mathcal{L})$.
\end{corollary}

\begin{proof}
If $i\not\leq_J \mathcal{J}(\mathcal{L})$, then, by construction, $a_i$ annihilates $C_{\mathcal{L}}$
and hence also $L_{\mathcal{L}}$. 

If $i\leq_J \mathcal{J}(\mathcal{L})$, then $a_i$ 
does not annihilate $C_{\mathcal{L}}$. Consider some $\mathbf{c}\in(\mathbb{R}_{>0}\cap \Bbbk)^n$
and the corresponding $a(\mathbf{c})$ as in \eqref{eq9}. Let $v(\mathbf{c})$ be a Perron-Frobenius
eigenvector for $a(\mathbf{c})$ in $C_{\mathcal{L}}$. Then $v(\mathbf{c})$ is a linear combination of
elements in $\mathbf{B}_{\mathcal{L}}$ with positive coefficients. As $A$ is positively based and
$i\leq_J \mathcal{J}(\mathcal{L})$, the element $a_i v(\mathbf{c})$ is a non-zero linear combination of 
elements in $\mathbf{B}_{\mathcal{L}}$ with non-negative coefficients. By construction,
$a_i v(\mathbf{c})\in Av(\mathbf{c})=V_{\mathcal{L}}$. Now, that the image of 
$a_i v(\mathbf{c})$ in $L_{\mathcal{L}}$ is non-zero, follows from Proposition~\ref{prop51}.
\end{proof}

\subsection{The apex for idempotent $\mathcal{J}$-cells}\label{s8.4}

\begin{proposition}\label{propnn23}
For an idempotent two-sided cell $\mathcal{I}$, we have $\mathcal{J}(\mathcal{I})=\mathcal{I}$.
\end{proposition}

\begin{proof}
Without loss of generality we may assume that $\mathcal{I}$ is maximal with respect to $\leq_J$
since all $a_i$ with $i>_J \mathcal{I}$ annihilate $C_{\mathcal{L}}$ by definition.
As $\mathcal{I}$ is idempotent, the set $\mathcal{I}\star \mathcal{I}$ is non-empty and hence 
coincides with $\mathcal{I}$, because of the maximality of the latter. Let 
$\mathcal{L}$ be a left cell in $\mathcal{I}$ which is minimal with respect to 
$\leq_L$. Then $\mathcal{L}\subset \mathcal{I}\star \mathcal{I}$. We have 
\begin{displaymath}
\mathcal{L}\subset \mathcal{I}\star \mathcal{I}= 
\mathcal{I}\star\big(\mathcal{L}\cup (\mathcal{I}\setminus\mathcal{L})\big)=
(\mathcal{I}\star\mathcal{L})\cup
(\mathcal{I}\star (\mathcal{I}\setminus\mathcal{L})).
\end{displaymath}
Because of the minimality of $\mathcal{L}$, it cannot have common elements with 
$\mathcal{I}\star(\mathcal{I}\setminus\mathcal{L})$. Therefore
$\mathcal{L}\subset \mathcal{I}\star\mathcal{L}$ which implies 
$\mathcal{J}(\mathcal{L})=\mathcal{I}$. Now the claim 
of our proposition follows from Proposition~\ref{prop31}.
\end{proof}

\section{Transitive modules and classification of special modules}\label{s9}

\subsection{Positively based modules}\label{s9.1}

Let $(A,\mathbf{B})$ be a positively based algebra and $(V,\mathbf{v})$ a based $A$-module. 
We will say that $(V,\mathbf{v})$ is {\em positively based} provided that, for any $a_i\in\mathbf{B}$
and any $v_s\in \mathbf{v}$, the element $a_i\cdot v_s\in V$ is a linear combination of elements in 
$\mathbf{v}$ with non-negative real coefficients. For example, the left regular $A$-module
${}_AA$ is positively based with respect to the basis $\mathbf{B}$.

For $v_s,v_t\in \mathbf{v}$, we write $v_s\boldsymbol{\to}v_t$ provided that there is $a_i\in \mathbf{B}$
such that the coefficient at $v_t$ in $a_i\cdot v_s$ is non-zero. The relation $\boldsymbol{\to}$ is,
clearly, reflexive and transitive. A based $A$-module $(V,\mathbf{v})$ will be called {\em transitive}
provided that $\boldsymbol{\to}$ is the full relation. For example, each $C_{\mathcal{L}}$, where
$\mathcal{L}$ is a left cell, is a transitive $A$-module with respect to the basis  $\mathbf{B}_{\mathcal{L}}$.

An interesting question seems to be how to decide whether a given $A$-module has a basis which makes
this module into a transitive module. 

\subsection{Special modules for transitive representations}\label{s9.2}

Let $(V,\mathbf{v})$ be a transitive $A$-module. Then, for every $\mathbf{c}\in(\mathbb{R}_{>0}\cap \Bbbk)^n$,
the corresponding element $a(\mathbf{c})$ from \eqref{eq9} is a Perron-Frobenius element for $(V,\mathbf{v})$. 
Therefore we can define the corresponding special subquotient $L(V,\mathbf{v},\mathbf{c})$. Exactly the
same argument as in the proof of Theorem~\ref{thm5} shows that $L(V,\mathbf{v},\mathbf{c})$ is independent
of $\mathbf{c}$. Hence we may denote $L(V,\mathbf{v},\mathbf{c})$ simply by $L_{(V,\mathbf{v})}$.

Similarly to Section~\ref{s8}, we can define the notion of the {\em apex} for any transitive 
$A$-module and appropriate versions of Propositions~\ref{prop1-1} and \ref{prop51} and also of 
Corollary~\ref{cor32} remain true for any transitive $A$-module.

\subsection{Idempotents related to the apex}\label{s9.3}

Let $\mathcal{L}$ be a left cell and $\mathcal{J}:=\mathcal{J}(\mathcal{L})$. 
Then $\mathcal{J}$ is an idempotent two-sided cell in $(\mathbf{n},\star)$.
Let $I_{\mathcal{J}}$ be the linear span in $A$ of all $a_i$ such that $i\not\leq_J \mathcal{J}$. 
Then $I_{\mathcal{J}}$ is an ideal in $A$. The quotient  $A_{\mathcal{J}}:=A/I_{\mathcal{J}}$ is
positively based with respect to the basis $\mathbf{B}_{\mathcal{J}}$ consisting of images  
$\overline{a}_j$ in $A_{\mathcal{J}}$ of all the $a_j$ for which $j\leq_J \mathcal{J}$. 
Denote by $\mathbf{I}$ the two-sided ideal of $A_{\mathcal{J}}$ spanned by 
$\overline{a}_j$, where $j\in \mathcal{J}$.

\begin{proposition}\label{prop55}
{\hspace{2mm}}

\begin{enumerate}[$($i$)$]
\item\label{prop55.1} There is an idempotent $e\in A_{\mathcal{J}}$ which can be written 
as a linear combination of $\overline{a}_j$, for $j\in\mathcal{J}$, with positive real coefficients. 
\item\label{prop55.2}
The idempotent  $e$ is primitive (for $A_{\mathcal{J}}$).
\item\label{prop55.3} 
The simple top of $A_{\mathcal{J}}e$ is isomorphic to $L_{\mathcal{L}}$.
\end{enumerate}
\end{proposition}

\begin{proof}
Consider the $A$-module $C_{\mathcal{L}}$. As $I_{\mathcal{J}}\cdot C_{\mathcal{L}}=0$, 
the $A$-module $C_{\mathcal{L}}$ is also, naturally, an $A_{\mathcal{J}}$-module. Let $a$ be the sum of
all $\overline{a}_j$, where $j\in\mathcal{J}$. Let  $[a]$ be the matrix of the action of 
$a$ on $C_{\mathcal{L}}$ in the basis $\mathbf{B}_{\mathcal{L}}$.

We claim that all entries in $[a]$ are positive.
First of all, we claim that all columns in $[a]$ are non-zero. Indeed, if $[a]$ has a zero column, then
there is $i\in\mathcal{L}$ such that $\overline{a}_ja_i=0$, for all $j\in \mathcal{J}$. This means
that $\mathbf{I}a_i=0$ and hence $\mathbf{I}A_{\mathcal{J}} a_i=0$ as $\mathbf{I}$ is a two-sided ideal
in $A_{\mathcal{J}}$. However, each $a_s$, where $s\in\mathcal{L}$, appears with a non-zero coefficient
in $u a_i$, for some $u\in A_{\mathcal{J}}$, by transitivity of $C_{\mathcal{L}}$. Therefore we
must have  $\mathbf{I}C_{\mathcal{L}}=0$ which contradicts $\mathcal{J}=\mathcal{J}(\mathcal{L})$.
This shows that each  column in $[a]$ is non-zero.

Next we claim that all entries in every column in $[a]$
are non-zero. Consider the column corresponding to $a_j$, for some $j\in\mathcal{L}$.
Then $aa_j\neq 0$ by the previous paragraph. Let 
$X\subset \mathcal{L}$ be the set of all those ${a}_i$ which appear in $a{a}_j$ with non-zero coefficients.
On the one hand, $X$ is non-empty. On the other hand, the fact that $\mathbf{I}$
is a two-sided ideal in $A_{\mathcal{J}}$,
implies that the linear span of $X$ is $A_{\mathcal{J}}$-invariant.
From the transitivity of $C_{\mathcal{L}}$, it thus follows that 
$X$ contains all  ${a}_i$, where $i\in\mathcal{L}$. Therefore all entires in  $[a]$ are positive.
Let $\lambda$ be the Perron-Frobenius eigenvalue of $[a]$.

Let us now assume that $\Bbbk=\mathbb{R}$, all other cases can be dealt with by restriction and extension
of scalars. Consider the matrix
\begin{displaymath}
M:=\lim_{m\to\infty}\frac{[a]^m}{\lambda^{m}}.
\end{displaymath}
From Theorem~\ref{thm2}\eqref{thm2.6}, it follows that $M$ is positive and $M^2=M$.  

For $m\geq 1$, consider the element
\begin{displaymath}
\frac{a^m}{\lambda^{m}}=\sum_{i\in\mathcal{J}} c_{i,m} \overline{a}_i,
\end{displaymath}
where $c_{i,m}\in\mathbb{R}_{> 0}$. As the matrix of the action of each $\overline{a}_i$
on $C_{\mathcal{L}}$ is non-zero and has non-negative entries, from the existence of the limit $M$ it follows that,
for each $i\in\mathcal{J}$, the sequence $\{c_{i,m}\,:\,m\geq 1\}$ converges, say to some
$c_i\in \mathbb{R}_{\geq 0}$. Let 
\begin{displaymath}
e=\sum_{i\in\mathcal{J}} c_{i} \overline{a}_i.
\end{displaymath}
Then $e^2=e$ and $M$ is the matrix of the action of $e$ on $C_{\mathcal{L}}$.

Since $1$ is a simple eigenvalue for $e$ by the Perron-Frobenius Theorem,
the left projective $A_{\mathcal{J}}$-module $A_{\mathcal{J}}e$ has simple top (cf. Subsection~\ref{s7.1}
and the corresponding property for $V_{\mathcal{L}}$). This means that $e$ is primitive,
proving claim~\eqref{prop55.2}.

To prove claim~\eqref{prop55.3}, it is enough to show that $e$ does not annihilate $L_{\mathcal{L}}$.
This follows by combining Proposition~\ref{prop51} with Corollary~\ref{cor32}.

To prove claim~\eqref{prop55.1}, it remains to show that all $c_i$ are non-zero.
Let $\mathcal{L}_1,\mathcal{L}_2,\dots,\mathcal{L}_k$ be an ordering of the left cells in 
$\mathcal{J}$ such that $\mathcal{L}_i\geq_L \mathcal{L}_j$ implies $i\leq j$, for all $i,j$.
Let $N$ be the matrix of the action of $e$ on $\mathbf{I}$ in the basis $\overline{a}_i$, where $i\in\mathcal{J}$,
which is ordered respecting the ordering of the $\mathcal{L}_p$'s. Then, combining
Propositions~\ref{prop31}, \ref{propnn23} and the arguments in the first part of the proof, we see that 
the matrix $N$ has the upper-triangular block form
\begin{displaymath}
\left(
\begin{array}{ccccc}
N_1 & * & * &\dots &*\\
0 & N_2 & * &\dots &*\\
0 & 0 & N_3 &\dots &*\\
\vdots & \vdots & \vdots &\ddots &\vdots\\
0 & 0 & 0 &\dots &N_k
\end{array}
\right),
\end{displaymath}
where each $N_p$ is a  positive matrix. 
As $N$ is idempotent, from \cite[Theorem~2]{Fl} it follows
that all off-diagonal blocks of $N$ are zero. Therefore $N$ is a direct sum of 
positive idempotent matrices. From Theorem~\ref{thm2}\eqref{thm2.4} and \eqref{thm2.5}
it follows that the only (up to scalar) non-negative eigenvector of $N$ has positive coefficients.
This means that all $c_i$ are positive and completes the proof.
\end{proof}

For a two-sided cell $\mathcal{J}$, let $X$ denote the $\Bbbk$-span in $A$ of all 
$a_i$, where $i\geq_J \mathcal{J}$ and $Y$ denote the $\Bbbk$-span in $A$ of all 
$a_i$, where $i>_J \mathcal{J}$. Set $M(\mathcal{J})=X/Y$, which is, naturally, a 
left $A$-module.

The following result follows from the proof of Proposition~\ref{prop55}.
It is interesting because it, in combination with Proposition~\ref{propnn23}, 
in particular, provides an elementary explanation for the corresponding phenomenon 
for Hecke algebras (where all two-sided cells are idempotent), see \cite[(5.1.13)]{Lu5}
and the remark after it.

\begin{corollary}\label{cornn53}
Let $\mathcal{I}$ be a two-sided cell and $\mathcal{J}=\mathcal{J}(\mathcal{I})$. 
Assume that $M(\mathcal{I})$ is an $A_{\mathcal{J}}$-module. Then all left cells in
$\mathcal{I}$ are not comparable with respect to the left order.
\end{corollary}

\begin{proof}
Let $\mathcal{L}_1,\mathcal{L}_2,\dots,\mathcal{L}_k$ be an ordering of the left cells in 
$\mathcal{I}$ such that $\mathcal{L}_i\geq_L \mathcal{L}_j$ implies $i\leq j$, for all $i,j$.
Consider the idempotent $e\in A_{\mathcal{J}}$ constructed in the proof of Proposition~\ref{prop55}. 
As $M(\mathcal{I})$ is an $A_{\mathcal{J}}$-module by assumption, the action of
$e$ on $M(\mathcal{I})$ is well-defined. Let $N$ be the matrix of this action, namely, the matrix 
of multiplicities of $a_i$, where $i\in \mathcal{I}$, in 
$a a_j$, where $j\in \mathcal{I}$, written similarly to the proof of Proposition~\ref{prop55}.
Similarly to the arguments in  the proof of Proposition~\ref{prop55}, $N$ is an idempotent
non-negative upper-triangular matrix with positive diagonal blocks and hence 
it must be a direct sum of positive idempotent matrices by \cite[Theorem~2]{Fl}.

Fix $q\in\{1,2,\dots,k\}$. From the diagonal form of $N$, we have that,
for any  $i\in\mathcal{L}_q$, all 
${a}_j$, where $j\in\mathcal{L}_q$,  appear with non-zero coefficients in 
elements of the form $u{a}_i$, where $u\in\mathbf{I}$. Moreover, no other 
${a}_s$ appear in this way.

Let now $j\in\mathcal{L}_q$ and ${a}_s$ be arbitrary. If
some ${a}_t$ appears with a non-zero coefficient in $\overline{a}_s{a}_j$,
it also appears with a non-zero coefficient in $\overline{a}_su{a}_i$, where
$i$ is as in the previous paragraph and for some
$u\in \mathbf{I}$. As $\mathbf{I}$ is a two-sided
ideal of $A_{\mathcal{J}}$, from the previous paragraph it follows that $t\in \mathcal{L}_q$.
The claim follows.
\end{proof}

In \cite[Subsection~4.3]{CM}, a two-sided cell $\mathcal{J}$ is called {\em good} provided that there 
is a linear combination $a$ of all $\overline{a}_j$, where $j\in\mathcal{J}$, with 
positive real coefficients, such that  
\begin{displaymath}
a^n+v_{n-1}a^{n-1}+\dots +v_{k+1}a^{k+1}=v_k a^k+v_{k-1} a^{k-1}+\dots+ v_l a^l
\end{displaymath}
for some $n,k,l\in\{1,2,\dots\}$ and some 
non-negative real numbers $v_{n-1}$, $v_{n-2},\dots$, $v_l$ such that $v_l\neq 0$.

\begin{corollary}\label{cornn54}
Let $\mathcal{L}$ be a left cell and $\mathcal{J}=\mathcal{J}(\mathcal{L})$.
Then $\mathcal{J}$ is good.
\end{corollary}

\begin{proof}
We can take $a=e$ which satisfies $a^2=a$.
\end{proof}

\subsection{Classification of special modules}\label{s9.4}

Now we are ready to classify all special modules appearing in all transitive modules.

\begin{theorem}\label{thmclass}
Let $(V,\mathbf{v})$ be a transitive $A$-module with apex $\mathcal{J}$. Then 
$L_{(V,\mathbf{v})}\cong L_{\mathcal{L}}$, for any left cell $\mathcal{L}$ in $\mathcal{J}$.
\end{theorem}

\begin{proof}
We may assume $A=A_{\mathcal{J}}$.
Let $\mathcal{L}$ be a left cell in $\mathcal{J}$ which is maximal with respect to $\leq_L$.
Let $e$ be an idempotent given by Proposition~\ref{prop55}. From the transitive versions of
Proposition~\ref{prop51} and  
Corollary~\ref{cor32} it follows that $e$ does not annihilate $L_{(V,\mathbf{v})}$.
As $e$ is primitive by Proposition~\ref{prop55}\eqref{prop55.3}, it follows that 
$L_{(V,\mathbf{v})}\cong L_{\mathcal{L}}$.
\end{proof}

As an immedeate consequence from Theorem~\ref{thmclass}, we have:

\begin{corollary}\label{thmclass-1}
Let $(V,\mathbf{v})$ be a transitive $A$-module with apex $\mathcal{J}$. Then 
$L_{(V,\mathbf{v})}$ does not depend on $\mathbf{v}$.
\end{corollary}

\begin{proof}
The apex $\mathcal{J}$ of $(V,\mathbf{v})$ does not depend on $\mathbf{v}$
and hence neither does the special module $L_{(V,\mathbf{v})}\cong L_{\mathcal{L}}$.
\end{proof}

Therefore, to list all special $A$-modules one has to do the following:
\begin{itemize}
\item identify all idempotent two-sided cells;
\item in each idempotent two-sided cell $\mathcal{J}$ fix a left cell, 
maximal with respect to $\leq_L$ among all left cells in $\mathcal{J}$;
\item compute the corresponding primitive idempotent $e$ for $A_{\mathcal{J}}$;
\item the corresponding special module is $A_{\mathcal{J}}e/\mathrm{Rad}(A_{\mathcal{J}}e)$.
\end{itemize}

Let us call a simple $A$-module {\em special} if it is isomorphic to a special
module for some transitive $A$-module. As an immediate corollary from the above, we have:

\begin{corollary}\label{corclass}
The above defines a one-to-one correspondence between the set of isomorphism classes of
special $A$-modules and the set of idempotent two-sided cells for $A$.
\end{corollary}

We do not know whether, for a non-semi-simple $A$, a cell module $C_{\mathcal{L}}$ might contain
$L_{\mathcal{L}'}$ for some left cell $\mathcal{L}'$ such that $\mathcal{J}(\mathcal{L}')\neq 
\mathcal{J}(\mathcal{L})$. Similarly, we do not know whether, for some $A$, a transitive $A$-module $(V,\mathbf{v})$ 
with $L_{(V,\mathbf{v})}\cong L_{\mathcal{L}}$ might contain $L_{\mathcal{L}'}$ such that 
$\mathcal{J}(\mathcal{L}')\neq \mathcal{J}(\mathcal{L})$.

Another interesting question is how to decide whether a given simple module $V$ over a positively 
based algebra  $A$ is special or not.

\noindent
Tobias Kildetoft, Department of Mathematics, Uppsala University,
Box 480,\\ SE-751~06, Uppsala, SWEDEN, {\tt tobias.kildetoft\symbol{64}math.uu.se}

\noindent
Volodymyr Mazorchuk, Department of Mathematics, Uppsala University,
Box 480, SE-751~06, Uppsala, SWEDEN, {\tt mazor\symbol{64}math.uu.se}


\begin{thebibliography}{99999}
\bibitem[CM]{CM} A.~Chan, V.~Mazorchuk. Diagrams and discrete extensions for 
finitary $2$-representations. Preprint arXiv:1601.00080.
\bibitem[EW]{EW} B.~Elias, G.~Williamson. The Hodge theory of Soergel bimodules. 
Ann. of Math. (2) {\bf 180} (2014), no. 3, 1089--1136. 
\bibitem[Fl]{Fl} P.~Flor. On groups of non-negative matrices. Compositio Math. {\bf 21} (1969), 376--382.
\bibitem[Fr1]{Fr1} G. Frobenius. \"{U}ber Matrixen aus positiven Elementen.
1. Sitzungsber. K\"{o}nigl. Preuss. Akad. Wiss. 1908, 471--476.
\bibitem[Fr2]{Fr2} G. Frobenius. \"{U}ber Matrixen aus positiven Elementen.
2. Sitzungsber. K\"{o}nigl. Preuss. Akad. Wiss. 1909, 514--518.
\bibitem[Fo]{Fo} L.~Forsberg. Multisemigroups with multiplicities and complete ordered semi-rings.
Preprint arXiv:1510.01478.
\bibitem[GM]{GM} O.~Ganyushkin, V.~Mazorchuk. Classical finite transformation semigroups. 
An introduction. Algebra and Applications {\bf 9}. Springer-Verlag London, Ltd., London, 2009.
\bibitem[GMS]{GMS} O.~Ganyushkin, V.~Mazorchuk, B.~Steinberg. On the irreducible 
representations of a finite semigroup. Proc. Amer. Math. Soc. {\bf 137} (2009), no. 11, 3585--3592.
\bibitem[Ge1]{Ge1} M.~Geck. Left cells and constructible representations. Represent. Theory 
{\bf 9} (2005), 385--416 
\bibitem[Ge2]{Ge2} M.~Geck. On the Kazhdan-Lusztig order on cells and families. 
Comment. Math. Helv. {\bf 87} (2012), no. 4, 905--927.
\bibitem[Gr]{Gr} J.~Green. On the structure of semigroups. Ann. of Math. (2) {\bf 54}, (1951). 163--172.
\bibitem[KL]{KL} D.~Kazhdan, G.~Lusztig. Representations of Coxeter groups and Hecke algebras. 
Invent. Math. {\bf 53} (1979), no. 2, 165--184.
\bibitem[KuMa]{KuMa} G.~Kudryavtseva, V.~Mazorchuk. On multisemigroups. Port. Math. {\bf 72} (2015), no. 1, 47--80.
\bibitem[Lu1]{Lu0} G.~Lusztig. Irreducible representations of finite classical groups. 
Invent. Math. {\bf 43} (1977), no. 2, 125--175.
\bibitem[Lu2]{Lu1} G.~Lusztig. A class of irreducible representations of a Weyl group.
Nederl. Akad. Wetensch. Indag. Math. {\bf 41} (1979), no. 3, 323--335. 
\bibitem[Lu3]{Lu2} G.~Lusztig. A class of irreducible representations of a Weyl group. II.
Nederl. Akad. Wetensch. Indag. Math. {\bf 44} (1982), no. 2, 219--226.
\bibitem[Lu4]{Lu3} G.~Lusztig. Sur les cellules gauches des groupes de Weyl, C. R. Acad. Sci. Paris 
{\bf 302} (1986), 5--8.
\bibitem[Lu5]{Lu5} G.~Lusztig. Characters of reductive groups over a finite field. Annals of 
Mathematics Studies {\bf 107}. Princeton University Press, Princeton, NJ, 1984.
\bibitem[MM1]{MM1} V.~Mazorchuk, V.~Miemietz. Cell $2$-representations of finitary
$2$-categories. Compositio Math. {\bf 147} (2011), 1519--1545.
\bibitem[MM2]{MM2} V.~Mazorchuk, V.~Miemietz. Additive versus abelian $2$-representations of 
fiat $2$-ca\-te\-go\-ri\-es. Moscow Math. J. {\bf 14} (2014), no. 3, 595--615.
\bibitem[MM3]{MM3} V.~Mazorchuk, V.~Miemietz. Endomorphisms of cell $2$-representations. 
Preprint arXiv:1207.6236.
\bibitem[MM4]{MM4} V.~Mazorchuk, V.~Miemietz. Morita theory for finitary $2$-categories. 
Preprint arXiv:1304.4698. To appear in Quantum Topol.
\bibitem[MM5]{MM5} V.~Mazorchuk, V.~Miemietz. Transitive $2$-representations of finitary 
$2$-categories. Preprint arXiv:1404.7589. To appear in Trans. Amer. Math. Soc.
\bibitem[MM6]{MM6} V.~Mazorchuk, V.~Miemietz. Isotypic faithful $2$-representations of 
$\mathcal{J}$-simple fiat $2$-categories. Preprint arXiv:1408.6102. To appear in Math. Z.
\bibitem[MZ]{MZ} V.~Mazorchuk, X.~Zhang. Simple transitive $2$-representations for two 
non-fiat $2$-categories of projective functors. Preprint arXiv:1601.00097.
\bibitem[Mu]{Mu} W.~Munn. Matrix representations of semigroups. Proc. 
Cambridge Philos. Soc. {\bf 53} (1957), 5--12.
\bibitem[Pe]{Pe} O. Perron. Zur Theorie der Matrices. 
Mathematischen Annalen \textbf{64} (2) (1907), 248--263.
\bibitem[Sch]{Sch} J.~Schauder. Der Fixpunktsatz in Funktionalr{\"a}umen. 
Studia Math. {\bf 2} (1930), 171--180.
\bibitem[So]{So} W.~Soergel. Kazhdan-Lusztig-Polynome und eine Kombinatorik f{\"u}r 
Kipp-Moduln. Represent. Theory {\bf 1} (1997), 37--68. 
\bibitem[Th]{Th} D.~Thurston. Positive basis for surface skein algebras. Proc. Natl. Acad. 
Sci. USA {\bf 111} (2014), no. 27, 9725--9732.
\bibitem[Zi]{Zi} J.~Zimmermann. Simile transitive $2$-representations of 
Soergel bimodules in type $B_2$. Preprint arXiv:1509.01441.
\end{thebibliography}
\end{document}